\documentclass[11pt]{article}
\usepackage[margin=1in]{geometry}
\usepackage{mathpazo}
\usepackage{amsmath,amssymb,amsthm,booktabs,tikz,xcolor}
\usetikzlibrary{calc,positioning}

\usepackage{amssymb}
\usepackage{amscd,amsfonts,amsmath,amssymb, bbm,dsfont,extarrows}
\usepackage{amsmath}
\usepackage{amsfonts}
\usepackage{amssymb}\usepackage{enumerate,epsf,fancyhdr,float,graphicx,tabularx}
\usepackage{latexsym,mathrsfs,multirow}
\usepackage{wasysym, microtype,needspace}
\usepackage{xypic}
\usepackage[hidelinks]{hyperref}
\usepackage[all]{xy}\usepackage[OT2,T1]{fontenc}
\usepackage[modulo,mathlines,displaymath, running]{lineno}
\usepackage{amsmath,mathrsfs,enumerate,enumitem,wasysym}
\usepackage{xypic,hhline}
\usepackage[all]{xy}
\usepackage[OT2,T1]{fontenc}
\usepackage[modulo,mathlines,displaymath]{lineno}
\usepackage{tikz,tikz-cd}
\usepackage{dashrule}
\usetikzlibrary{matrix}
\usepackage[modulo,mathlines,displaymath]{lineno}

\newtheorem{theorem}{Theorem}[section]
\DeclareSymbolFont{cyrletters}{OT2}{wncyr}{m}{n}\DeclareMathSymbol{\Sha}{\mathalpha}{cyrletters}{"58}

\renewcommand{\phi}{{\varphi}}

\renewcommand{\geq}{\geqslant}
\renewcommand{\leq}{\leqslant}

\newcommand{\links}{\left(\begin{array}{cc}}
\newcommand{\rechts}{\end{array}\right)}
\newcommand{\bai}{\left[\begin{array}{cc}}
\newcommand{\dai}{\end{array}\right]}
\newcommand{\hidari}{\left(\begin{array}{c}}
\newcommand{\migi}{\end{array}\right)}

\newcommand{\fr}{\mathcal F}
\newcommand{\G}{{\mathcal G}}

\renewcommand{\contentsname}{Contents\\{\footnotesize\normalfont(A table
of contents should normally not be included)}}

\newtheorem{auxiliary proposition}[theorem]{Auxiliary Proposition}

\newtheorem{corollary}[theorem]{Corollary}

\newtheorem{definition}[theorem]{Definition}

\newtheorem{lemma}[theorem]{Lemma}
\newtheorem{main conjecture}[theorem]{Main Conjecture}
\newtheorem{main theorem}[theorem]{Main Theorem}
\newtheorem{modesty proposition}[theorem]{Modesty Proposition}

\newtheorem{observation}[theorem]{Observation}
\newtheorem{open problem}[theorem]{Open Problem}

\newtheorem{proposition}[theorem]{Proposition}

\newtheorem{convergence lemma}[theorem]{Convergence Lemma}
\newtheorem{corrected lemma}[theorem]{Corrected Lemma}
\newtheorem{growth lemma}[theorem]{Growth Lemma}
\newtheorem{coefficient lemma}[theorem]{Integrality Lemma}
\newtheorem{interpolation lemma}[theorem]{Interpolation Lemma}
\newtheorem{kernel lemma}[theorem]{Kernel Lemma}
\newtheorem{limit lemma}[theorem]{Limit Lemma}
\newtheorem{tandem lemma}[theorem]{Modesty Lemma}
\newtheorem{zero-finding lemma}[theorem]{Zero-Finding Lemma}

\newcommand{\firstdiamond}{\begin{tikzpicture}[baseline=-2pt,x=0.29cm,y=0.22cm,every node/.style={inner sep=0pt,font=\scriptsize}]
\node at (-1,0) {$a$};\node at (0,1) {$b$};\node at (1,0) {$c$};\node at (0,-1) {$d$};
\end{tikzpicture}}
\newcommand{\ordinarydiamond}{\begin{tikzpicture}[baseline=-2pt,x=0.30cm,y=0.23cm,every node/.style={inner sep=0pt,font=\scriptsize}]
\node at (-1,0) {$a$};\node at (0,1) {$b$};\node at (1,0) {$d$};\node at (0,-1) {$c$};
\end{tikzpicture}}
\newcommand{\forkdiamond}{\begin{tikzpicture}[baseline=-2pt,x=0.33cm,y=0.26cm,every node/.style={inner sep=0pt,font=\scriptsize}]
\node at (-1,0) {$e$};\node at (0,1) {$f$};\node at (0,0) {$g$};\node at (0,-1) {$h$};\node at (1,0) {$i$};
\end{tikzpicture}}
\newcommand{\onediamond}{\begin{tikzpicture}[baseline=-2pt,x=0.33cm,y=0.26cm,every node/.style={inner sep=0pt,font=\scriptsize}]
\node at (-1,0) {*};\node at (1,0) {*};\node at (0,1) {$1$};\node at (-1,-1) {$j$};\node at (0,-1) {$k$};\node at (1,-1) {$\ell$};
\end{tikzpicture}}

\newcommand{\friezearray}{%
\begin{tikzpicture}[x=1.10cm,y=0.51cm,every node/.style={inner sep=1pt,font=\small}]
\foreach \x in {0,2,4,6,8} {\node at (\x,0) {$1$};}
\node at (-1,0) {$\cdots$};
\foreach \x/\j in {1/1,3/2,5/3,7/4}{\node at (\x,-1) {$x_{2,\j}$};}
\node at (-.5,-1) {$\cdots$};
\foreach \x/\j in {0/0,2/1,4/2,6/3,8/4}{\node at (\x,-2) {$x_{3,\j}$};}
\foreach \x/\j in {1/1,3/2,5/3,7/4}{\node at (\x,-2) {$x_{1,\j}$};}
\node at (-1,-2) {$\cdots$};
\foreach \x/\j in {1/0,3/1,5/2,7/3}{\node at (\x,-3) {$x_{4,\j}$};}
\node at (-.5,-3) {$\cdots$};
\foreach \x/\j in {0/{-1},2/0,4/1,6/2,8/3}{\node at (\x,-4) {$x_{5,\j}$};}
\node at (-1,-4) {$\cdots$};
\node at (1,-5) {$\vdots$};
\foreach \x/\j in {0/{-2},2/{-1},4/0,6/1,8/2}{\node at (\x,-6) {$x_{n,\j}$};}
\foreach \x in {1,3,5,7}{\node at (\x,-7) {$1$};}
\end{tikzpicture}}

\tikzset{msarc/.style={line width=.55pt},msvertex/.style={circle,fill=black,inner sep=1.35pt},mslabel/.style={inner sep=1pt,font=\small,fill=white}}
\newcommand{\exchangerules}{%
\begin{tikzpicture}[x=1cm,y=1cm,msarc,every node/.style={mslabel}]
\begin{scope}[shift={(1.3,0)}]
\coordinate (A) at (-.95,.95);\coordinate (B) at (.95,.95);
\coordinate (C) at (.95,-.95);\coordinate (D) at (-.95,-.95);
\draw (A)--node[above]{$a$}(B)--node[right]{$b$}(C)--node[below]{$c$}(D)--node[left]{$d$}(A);
\draw (A)--(C);\draw (D)--(B);
\node at (-.45,-.62) {$x$};\node at (.47,-.62) {$y$};
\foreach \p in {A,B,C,D}{\node[msvertex] at (\p) {};}
\node at (0,-1.58) {Ptolemy: $xy=ac+bd$.};
\end{scope}

\begin{scope}[shift={(5.15,0)}]
\draw (0,0) circle (1cm);
\coordinate (A) at (36:1);    
\coordinate (B) at (145.5:1); 
\draw (0,0)--(A);
\draw (0,0)--node[pos=.30,inner sep=0pt, font=\tiny] {$\bowtie$} (B);
\draw (0,0)..controls(-.05,.34)and(-.42,.62)..(B);
\foreach \p in {A,B}{\node[msvertex] at (\p) {};}
\node[msvertex] at (0,0) {};
\node at (-1.10,-.39) {$a$};
\node at (0.69,1.04) {$b$};
\node at (.33,.22) {$y$};
\node at (-.51,.11) {$x$};
\node at (-.50,.56) {$u$};
\end{scope}

\begin{scope}[shift={(8.45,0)}]
\draw (0,0) circle (1cm);
\coordinate (A) at (78:1);    
\coordinate (B) at (-28:1);  
\draw (0,0)--node[pos=.30,inner sep=0pt, font=\tiny] {$\bowtie$} (A);
\draw (0,0)--(B);
\draw (0,0)..controls(.08,-.30)and(.50,-.56)..(B)
  node[pos=.24,inner sep=0pt, font=\tiny] {$\bowtie$};
\foreach \p in {A,B}{\node[msvertex] at (\p) {};}
\node[msvertex] at (0,0) {};
\node at (-.96,.54) {$a$};
\node at (1.12,-.02) {$b$};
\node at (.07,.56) {$y$};
\node at (.45,-.16) {$x$};
\node at (.23,-.58) {$u$};
\end{scope}
\node at (6.80,-1.58) {Clock-hand rule: $xy=a+b$.};

\begin{scope}[shift={(12.35,0)}]
\coordinate (A) at (-.953,.55);
\coordinate (B) at (.953,.55);
\coordinate (C) at (0,-1.10);
\draw (0,0) circle (1.10cm);
\draw (A)..controls(.77,1.2)and(1.1,-.6)..(C);
\draw (B)..controls(-.77,1.2)and(-1.1,-.6)..(C);
\draw (0,0)--(C);
\draw (0,0)..controls(-.31,-.38)and(-.31,-.85)..(C);
\foreach \p in {A,B,C}{\node[msvertex] at (\p) {};}
\node[msvertex] at (0,0) {};
\node[font=\tiny,inner sep=0pt] at (0,-.24) {$\bowtie$};
\node at (0,1.32) {$a$};
\node at (-1.17,-.48) {$c$};
\node at (1.17,-.48) {$b$};
\node at (-.64,-.24) {$x$};
\node at (.64,-.24) {$y$};
\node at (-.41,-.62) {$u$};
\node at (.18,-.64) {$t$};
\node at (0,-1.58) {Candle-flame rule: $xy=bc+atu$.};
\end{scope}
\end{tikzpicture}}

\newcommand{\radiiquadrilateral}{%
\begin{tikzpicture}[x=.93cm,y=.93cm,msarc,every node/.style={mslabel}]
\coordinate (P) at (-1.2,0);\coordinate (A) at (0,.95);\coordinate (B) at (1.05,0);\coordinate (C) at (0,-.95);
\draw (P)--node[above left]{$r_{j-1}$}(A)--(B)--(C)--node[below left]{$r_{j+1}$}(P);
\draw (P)--node[above,pos=.38]{$r_j$}(B);\draw (A)--(C);
\foreach \p in {P,A,B,C}{\node[msvertex] at (\p) {};}
\node[left=3pt] at (P) {$p$};\node[above=3pt] at (A) {$v_{j-1}$};
\node[right=3pt] at (B) {$v_j$};\node[below=3pt] at (C) {$v_{j+1}$};
\draw (.08,.25)--(.65,.69);\node[above right] at (.61,.66) {$m_{j-1,j+1}$};
\end{tikzpicture}}

\newcommand{\comparisonDigon}{%
\begin{tikzpicture}[x=1cm,y=1cm,line cap=round,line width=.5pt,
 every node/.style={font=\small,inner sep=2pt}]
 \begin{scope}[shift={(-2.65,0)}]
  \coordinate (V) at (150:1.25); \coordinate (W) at (-20:1.25);
  \coordinate (P) at (0,0);
  \draw (0,0) circle (1.25);
  \draw (P)--(V);\draw (P)--(W);
  \node[font=\tiny,inner sep=0pt,fill=white] at (-20:.28) {$\bowtie$};
  \foreach \q in {V,W,P}{\fill (\q) circle (1.4pt);}
  \node[above left] at (V) {$v$};\node[below right] at (W) {$w$};
  \node[below left] at (P) {$p$};
  \node at (0,1.49) {$a$};\node at (0,-1.49) {$a'$};
  \node[fill=white] at (-.69,.09) {$r_v$};
  \node[fill=white] at (.68,-.61) {$r_w^{\bowtie}$};
  
 \end{scope}
\end{tikzpicture}}

\title{\bfseries How large can $B_n$- and $D_n$-friezes be?}
\author{ Florian Ito Sprung\footnote{Current address: 1 Oxford Street, Department of Mathematics, Harvard University, Cambridge MA 02138\\ E-mail: ian.sprung@gmail.com}}
\date{}
\begin{document}
\maketitle

\begin{abstract}
We pin down the largest entries possible in positive integral friezes
of types $D_n$ for $n\geq 4$ and $B_n$ for $n\geq 2$, addressing a conjecture of Robin Zhang. For type $D_n$, the sharp upper bound is $F_nF_{n+1}-1$, and for type $B_n$, it is $F_{n+1}F_{n+2}-1$, where $F_k$ are the Virahanka--Fibonacci numbers.
\end{abstract}

\section{Introduction}
The arithmetic of friezes goes back to at least the work of Conway and Coxeter \cite{ConwayCoxeter1973a,ConwayCoxeter1973b}, who developed the subject by posing (and answering) a number of questions on how one can fill out arrays of integers satisfying the \textit{Ptolemy relations}, where for each diamond \firstdiamond, we require $ac=bd+1$.

Conway and Coxeter established that their friezes corresponded to triangulations of polygons. Subsequently, their work was generalized by making a connection to Lie theory, more precisely to finite-type cluster algebras and their Dynkin diagrams. Conway and Coxeter's friezes correspond to Dynkin diagrams of type $A_n$, and friezes coming from other Dynkin diagrams were defined.

One natural question one might ask is how big the (positive integer) entries of a frieze may be. In the case of friezes of type $A_n$ (and also $C_n$), this question was settled by Cheah and de Saint Germain, who showed that the entries could attain but not exceed the Virahanka--Fibonacci numbers $F_{n+2}$ (resp.\ $F_{2n+1}$) \cite[Theorem~1]{CheahSaintGermain2025}, defined by $F_0=0$, $F_1=1$, and $F_{k+2}=F_{k+1}+F_k$.
For friezes of other types, there is now a vast web of results -- for example, the answer for the exceptional types $E_6$ and $G_2$ is $307$ and $14$, respectively \cite[Figure~2]{Zhang2026}. Despite a lot of effort by the experts, a sharp upper bound on the entries of friezes of type $D_n$ and $B_n$ is still outstanding. Here, Robin Zhang constructed $D_n$-friezes whose maximal entries were $F_nF_{n+1}-1$, and  $B_n$-friezes with maximal entries $F_{n+1}F_{n+2}-1$ \cite[Theorem~2]{Zhang2026}. In \cite[Conjecture~3]{Zhang2026}, Zhang conjectured that these were the largest possible entries in such friezes.

The point of this paper is to propose a proof of Robin Zhang's conjecture.

\begin{theorem}[Corollary \ref{cor:mainresult}]

Let $n\geq2$. Then every entry in a positive integral $B_n$-frieze is at most $F_{n+1}F_{n+2}-1$. Let $n\geq4$. Then every entry in a positive integral $D_n$-frieze is at most $F_nF_{n+1}-1$. Both bounds are sharp.
\end{theorem}


Although Dynkin friezes have their origin in the theory of cluster algebras due to Fomin and Zelevinsky \cite{FominZelevinsky2003}, they can be defined purely combinatorially\cite[Definition~2.1]{Zhang2026}:
\begin{definition}[Dynkin frieze]
Given a finite Dynkin type $\Delta_n$ with generalized Cartan matrix $C=(c_{i,j})_{1\le i,j\le n}$, a positive integral frieze of type $\Delta_n$ is an array of positive integers $(x_{i,j})_{1\le i\le n,\,j\in\mathbb Z}$ satisfying
\begin{equation}\label{eq:frieze}
 x_{i,j}x_{i,j+1}-\prod_{k<i}x_{k,j+1}^{|c_{k,i}|}\prod_{k>i}x_{k,j}^{|c_{k,i}|}=1
 \qquad\forall i\ \forall j.
\end{equation}
\end{definition}

Most relevant to our paper is the case $\Delta_n=D_n$. Here, the entries of $C$ are $c_{ii}=2, c_{j(j+1)}=c_{(j+1)j}=-1$ if $j\geq2$, and $c_{13}=c_{31}=-1$, while all other $c_{ij}=0$, and this array is conveniently represented as
\begin{center}
\friezearray
\end{center}

The subdiamonds \ordinarydiamond \quad,\quad  \forkdiamond \quad, and \onediamond are required to satisfy the equations $ad-bc=1$, $ei-fgh=1$, and $j\ell=1+k$. (The last diamond is the first diamond in disguise: the extra middle entries $x_{1,j}$ can be thought of as `alternatives' to the entries $x_{2,j}$ directly above them. We thus require the relations $x_{1,j}x_{1,j+1}=1+x_{3,j}$.)
For many examples of frieze patterns, see \cite{deSaintGermainVisualCA}.


The rough summary of the proof in the case of $D_n$-friezes is to cut open a triangulation associated to $D_n$, which we call the \textit{plafond triangulation}\footnote{
This weighted triangulation is due to Fontaine and Plamondon \cite[Proposition~3.2]{FontainePlamondon2016}, who built heavily on ideas of Baur and Marsh \cite[Theorem~1.1 and Definition~6.7]{BaurMarsh2009} and Thomas \cite[Proposition~A.2]{BaurMarsh2009} . The underlying geometric model was developed by Schiffler \cite[Section~5.2]{Schiffler2008} and \cite[Section~7]{FominShapiroThurston2008}. The term is the French word
for ``ceiling,'' so the ``d'' at the end of ``plafond'' is not pronounced.
} to make it amenable to the techniques of $A_n$-type friezes, and to then generalize Cheah and de Saint Germain's bounds to products of pairs of entries, which under the arithmetic of the triangulation give us the desired upper bound. 

\textbf{The proof strategy, along with an informal description of the plafond traingulation.}

Fontaine and Plamondon showed that each $D_n$-frieze corresponds to a plafond triangulation, which is a triangulation of the $n$-gon with a puncture inside, in analogy to the result of Conway and Coxeter. Conway and Coxeter ``unpack'' a triangulation of an $n$-gon following the Ptolemy rule to obtain an $A_{n-3}$-frieze. Fontaine and Plamondon, building on ideas of Baur, Marsh, and Thomas \cite{BaurMarsh2009}, generalized this to $D_n$-friezes and showed that a plafond triangulation can be similarly unpacked by following three rules. In addition to the Ptolemy rule, we also have rules that involve the puncture. These are the clock-hand rule and the candle-flame rule sketched below, in which the puncture is the point in the center:
\begin{center}
    \exchangerules
\end{center}    

The plafond triangulation always involves \textit{spokes}, each of which connects the puncture to a vertex of the $n$-gon. For the moment, let us assume that the label in the frieze along all spokes is $1$. (The other arcs in the plafond triangulation all have label $1$ by construction.). By cutting along this spoke, the puncture becomes a boundary point, and we may thus extract out of the $D_n$-frieze\footnote{The idea of ``cutting'' is already visible in \cite[proof of Lemma~3.4]{FontainePlamondon2016} in the form of ``compatible arcs'', and originally due to Hugh Thomas \cite[proof of Proposition~A.2]{BaurMarsh2009}.} an $n+2$-gon, giving us an $A_{n-1}$-frieze. We are interested in what happens to the \textit{plain radii} of the $D_n$-frieze after cutting into the $A_{n-1}$-frieze. (A radius is \textit{any} arc in the $D_n$-frieze that connects the puncture to a boundary point, not just the spokes. Spokes are radii that happen to be in the plafond triangulation. There are two types of radii, \textit{plain} and \textit{notched}.) Repeatedly applying the Ptolemy rule, we can see that the labels along these plain radii form an $A_{n-1}$-chain, i.e. the diagonal of an $A_{n-1}$-frieze\footnote{Since we cut along a spoke, which by assumption was labeled by a $1$, so that the first and last entry in the chain are both $1$s, which we truncate.}. 

Cheah and de Saint Germain gave the sharp upper bound $F_{m+2}$ for entries in an $A_m$-chain\footnote{Cheah and de Saint Germain's central idea is this: They showed that such a diagonal $(a_1,\ldots,a_m)$ has a special index $r$ (the \textbf{r}ight index) which satisfies $a_r=a_{r-1}+a_{r+1}$ when $a_m\neq 1$. Their bounds on the entries came from the observation that leaving out $a_r$ (or $a_m$, in case $a_m=1$) gives rise to a diagonal which is an $A_{m-1}$-chain. Thus, $a_r$ can be described as the sum of the entries $a_{r-1}$ and $a_{r+1}$ from the smaller $A_{m-1}$-chain. Proceeding by induction, one gets the bound of Virahanka--Fibonacci numbers: $a_j\le F_{m+2}$.}. This allows us to bound the labels of the arcs in our $D_n$-frieze that are radii (by $F_{n+1}$), but how about the other arcs, e.g. those arcs that connect boundary points? 
For the arcs connecting boundary points, the clock-hand rule above says that the sum of the label of an arc and its complementary arc (connecting the same boundary points but going around the puncture in the opposite way) is the product of the labels of the radii connecting the arc and the puncture, with one caveat: among the two radii involved, one must be plain, and one must be notched.
If we used Cheah and de Saint Germain's bound naively, these rules would give us bounds of the form $F_{m+2}^2$, which is too coarse for our purposes. We thus revisit their work \cite{CheahSaintGermain2025} and generalize their techniques to show that a product of two entries at distinct positions in an $A_m$-chain is bounded not only by $F_{m+2}^2$, but by $F_{m+2}F_{m+1}$ -- the bound is sharp. The clock-hand rule thus seems to imply that the sum of the label of an arc and its complement is bounded above by $F_{m+2}F_{m+1}$, and since the labels are positive, each label would be at most $F_{m+2}F_{m+1}-1$, which looks like Zhang's predicted bound. However, there is a problem with this argument. The $xy$-term in the clock-hand rule is a product of a plain label and a \textit{notched} label (remember the caveat). However, our $F_{m+2}F_{m+1}$-bound applies to products of plain radii. To address the labels of the notched radii, we need to dive a bit deeper. Denoting by $s$ the number of spokes in the plafond triangulation, Baur and Marsh \cite[Lemma 5.5]{BaurMarsh2009} essentially proved that the label of a notched radius is $s$ times the label of its plain counterpart, cf. Observation \ref{obs:tagexchange}. Thus, the upper bound on the $a+b$ term in the clock-hand rule is $sF_{m+2}F_{m+1}$, which seems hopelessly large at first glance\footnote{It may get worse: The radii may lie in different $A_m$-chains, potentially giving an even worse bound.}. One central non-obvious step gets us out of this rut: Since the $s$ spokes are each labeled by $1$, they split the $A_{n-1}$-chain into smaller $A_m$-chains -- in fact, $m$ is at most $n-s$, so that the upper bound is $sF_{n+2-s}F_{n+1-s}$ in the case of one big $A_{n-s}$-chain. It is not hard to show that this bound also holds when we have multiple $A_m$-chains because their total length is bounded by $n-s$. An elementary analysis of Virahanka-Fibonacci numbers shows $sF_{n+2-s}F_{n+1-s}\leq F_{n+1}F_n$, and positivity of the labels of the arcs implies each is bounded by $F_{n+1}F_n-1$, which is indeed Zhang's conjectured bound! The only remaining arcs are the notched radii, but note that we just showed they are bounded above by $sF_{m+2}\leq sF_{n+2-s}<F_{n+1}F_n$.

This resolves the conjecture when the spokes in the plafond triangulation were labeled by $1$. In the remaining cases, the plafond triangulation involves spokes labeled by the same integer $d$ --  required to \textbf{d}ivide the number $s$ of spokes. Fontaine and Plamondon compare such a plafond triangulation with its two extreme variations, i.e. the same plafond triangulations, but with all $1$'s (denoted $U_1$) or all $s$'s on the spokes ($U_s$), and show that the non-radii have the same labels \cite[Theorem 4.7 and sentence after Lemma 4.9]{FontainePlamondon2016}. This almost reduces the proof to the case of plafond triangulations with $1$'s on the spoke -- we only have to handle the radii. \cite[Lemma 4.8]{FontainePlamondon2016} seems to worsen our situation, since it says that that the label of a plain radius is $d$ times that of the same plain radius in $U_1$. What comes to the rescue is an even worse bound: In $U_1$, one can find a radius with an even worse label (it is the notched radius, and \ref{obs:tagexchange} shows it is worse since it is $s\geq d$ times as large as that of the plain radius). But we already handled the $U_1$ case, so this seemingly even worse label is bounded above by Zhang's predicted bound, so we are done. 


The result for $B_n$-friezes follows directly from that of $D_{n+1}$-friezes when $n\geq 3$. \cite[Lemma~4.1 and proof of Theorem 4.3]{FontainePlamondon2016} show that every positive integral $B_n$-frieze lifts to a positive integral $D_{n+1}$-frieze with no new entries. For $B_2$, the result is already known, as $B_2\cong C_2$ \cite[Theorem~1(2)]{CheahSaintGermain2025}.

\section{The plafond triangulation}\label{sec:plafond-details}

Conway and Coxeter constructed a bijection between frieze patterns of $A_n$-type and triangulations of the $n+3$-gon, see \cite[Questions]{ConwayCoxeter1973a}, \cite[Solutions]{ConwayCoxeter1973b}, and also \cite[Theorem~2.1]{FontainePlamondon2016}.

Analogously, Fontaine and Plamondon proved \cite[Proposition~3.2]{FontainePlamondon2016} that each $D_n$-frieze comes from a \textit{plafond triangulation}. These are defined when $n\geq4$, which we assume for the rest of the paper. There seem to be two differing conventions about triangulations in the literature of an $n$-gon. In one, the boundary edges are not included \cite{FontainePlamondon2016}\cite[Definitions~2.2 and 2.6]{FominShapiroThurston2008}\cite[Definitions~2.1--2.2]{BaurMarsh2009}\footnote{This is compatible with the fact that it is the interior arcs that correspond to cluster variables.}. In the other, the boundary edges are included \cite[Section~2]{SchifflerThomas2009}. We follow the convention to include the boundary edges, simply because it make our arguments slightly easier.

\paragraph{The Conway--Coxeter triangulation.}
Conway and Coxeter showed that any $A_n$-frieze comes from a
triangulation of an $n+3$-gon. More precisely, all the ``edges''
(which we call arcs) in the triangulation are labelled by $1$,
and the other labels appearing in the $A_n$-frieze are determined
by successive applications of the Ptolemy rule, described by the picture below (for a more visual explanation, watch  \cite{HaranTabachnikovFriezes}):

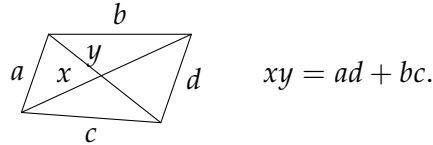
\begin{center}
\begin{tikzpicture}[scale=.9, line cap=round, line join=round]
  \coordinate (A) at (0,1.3);
  \coordinate (B) at (2.1,1.3);
  \coordinate (C) at (1.65,0);
  \coordinate (D) at (-.4,.15);

  \draw (A) -- node[above] {$b$} (B)
            -- node[right] {$d$} (C)
            -- node[below] {$c$} (D)
            -- node[left] {$a$} (A);

  \draw (D) -- node[pos=.25,above] {$x$} (B);
  \draw (A) -- node[pos=.25,right] {$y$} (C);

  \node at (4.4,.7) {$xy=ad+bc$.};
\end{tikzpicture}
\end{center}

\paragraph{The plafond triangulation.}
Fontaine and Plamondon showed that any $D_n$-frieze comes from
a plafond triangulation\cite[Proposition~3.2]{FontainePlamondon2016}. We make the convention that the plafond triangulation starts with an $n$-gon with a puncture and triangulates from the outside in. Start with the edges of the outer $n$-gon\footnote{\cite{FontainePlamondon2016} do not include the boundary edges, because only the inner edges correspond to cluster variables.} and label them with $1$. Successively repeat
the following: add one arc \footnote{`edge,' but we use the term arc, since it may be deformed by an isotopy fixing the marked points, see \cite[Section~3.1]{FontainePlamondon2016} and
\cite[Definition~2.1]{BaurMarsh2009}} that completes a triangle
(i.e.\ the arc is one side of the triangle, the other two sides
are already in the triangulation), and does not cross any
previous arcs. Continue doing this until only a connected $s$-gon around the
puncture is left, where for the moment, we assume $s\ge2$. 
Outside of the inner $s$-gon, the outer $n$-gon has been
triangulated. Note that the digon case $(s=2)$ is possible! Label all the arcs in the triangulation so far
by $1$.

Then we add $s$ spokes, connecting the puncture to each vertex of the inner $s$-gon. All spokes are labeled by the same positive integer
$d$, where we require that $d\mid s$.

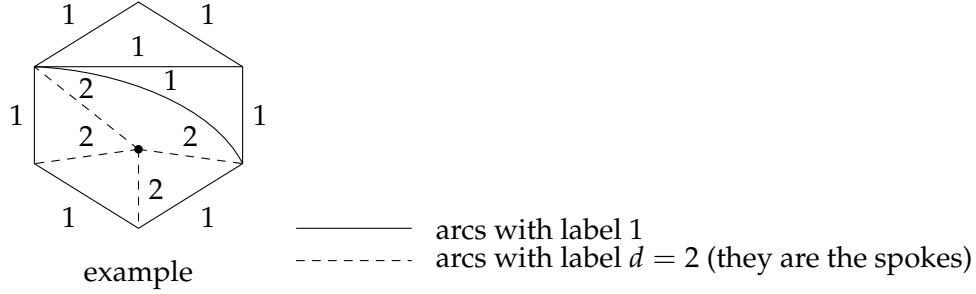
\begin{center}
\begin{tikzpicture}[scale=.95, line cap=round, line join=round]
  \coordinate (A) at (0,1.6);
  \coordinate (B) at (1.45,.7);
  \coordinate (C) at (1.45,-.65);
  \coordinate (D) at (0,-1.55);
  \coordinate (E) at (-1.45,-.65);
  \coordinate (F) at (-1.45,.7);
  \coordinate (P) at (0,-.45);

  \draw (A) -- node[above right] {$1$} (B)
            -- node[right] {$1$} (C)
            -- node[below right] {$1$} (D)
            -- node[below left] {$1$} (E)
            -- node[left] {$1$} (F)
            -- node[above left] {$1$} (A);

  \draw (F) -- node[above] {$1$} (B);
  \draw (F) .. controls (-.25,.65) and (1.05,.15) ..
    node[pos=.55,above] {$1$} (C);

  \draw[dashed] (P) -- node[midway,above] {$2$} (F);
  \draw[dashed] (P) -- node[midway,above] {$2$} (C);
  \draw[dashed] (P) -- node[midway,right] {$2$} (D);
  \draw[dashed] (P) -- node[midway,above] {$2$} (E);
  \fill (P) circle (1.7pt);

  \node[below] at (0,-1.85) {example};

  \node[anchor=west] at (4.0,-1.55) {arcs with label $1$};
  \draw (2.2,-1.55) -- (3.8,-1.55);

  \node[anchor=west] at (4.0,-2.0) {arcs with label $d=2$ (they are the spokes)};
  \draw[dashed] (2.2,-2.0) -- (3.8,-2.0);
\end{tikzpicture}
\end{center}

The  plafond triangulation concerning
the case $s=1$ is a bit different. We start with an $n$-gon, and pick one vertex. We draw a loop based at the vertex that encloses the puncture, then triangulate the complement. Then insert a spoke connecting the chosen vertex and the puncture (inside the loop). Finally, delete the loop from before and replace it by a notched copy of the spoke, see \cite[Section~3.3]{FontainePlamondon2016} and
\cite[Definition~6.6(i)]{BaurMarsh2009}. The labels of the arcs so far (including the plain and notched spoke) are $1$.


\textbf{$D_n$-friezes from plafond triangulations.}
The $D_n$-frieze consists of the labels of all \textit{admissible tagged} arcs connecting any pair of vertices of the punctured $n$-gon: These are all arcs except for self-intersecting paths, the boundary arcs, those paths that cut out unpunctured $1$-gons or digons, arcs whose interiors cross the boundary or marked points, or loops that enclose punctured $1$-gons. Arcs between the puncture and a boundary point are \textit{radii}. To fill out a $D_n$-frieze, we make the convention to produce two versions of each radius; we have plain radii and notched
radii \footnote{The tags are at the puncture and measure compatibility of arcs, which we do not discuss here. But see also \cite[Section~3.1]{FontainePlamondon2016}.}). 

The values of the arcs are then
determined by the exchange rules: the Ptolemy rule, as well as the clock-hand
rule and the candle-flame rule. We recall their diagrams, and the fact that the center dot in the clock-hand rule and the candle-flame rule is the puncture.

\begin{center}
\exchangerules
\end{center}

\begin{theorem}\cite[Theorem~3.1, Proposition~3.2, and proof of Theorem~3.9]{FontainePlamondon2016}
Successive applications of the exchange rules to any plafond triangulation gives rise to a $D_n$-frieze. Conversely, any $D_n$-frieze corresponds to a unique plafond triangulation.
\end{theorem}


For the convenience of the reader, we spell out what this means for the entries $x_{i,j}$ in the $D_n$-frieze diagram:
\begin{center}
\friezearray
\end{center}

Label the vertices on the boundary of the $n$-gon by $v_j$, clockwise.
For the moment, assume $i \geq 3$. The subscripts in $x_{i,j}$ mean ``$i-2$ steps down clockwise'' and ``connect towards the vertex $v_j$'': $x_{i,j}$ denotes the arc that connects vertex $v_{j+i-2}$ to vertex $v_j$\footnote{Note this convention differs from that of Baur and Marsh's, who call $m_{ij}$ the label of the arc that connects vertices $i$ and $j$\cite[Definition 2.14]{BaurMarsh2009}}, clockwise around the puncture. For example, $x_{3,j}$ is the long arc between adjacent vertices $v_{j+1}$ and $v_j$.
As for the $x_{i,j}$ with $i\in\{1,2\}$, these denote radii connecting the puncture to vertex $v_j$. The convention is that $x_{i,j}$ is a plain radius if $i+j$ is even and a notched radius if $i+j$ is odd (remember this is for $i\in\{1,2\}$). Explicitly, we have $$(x_{1,j},x_{2,j})=\begin{cases}(r_j,r_j^{\bowtie}), & j \text{odd},\\
 (r_j^{\bowtie},r_j), & j \text{even}. \end{cases}$$

\begin{observation}\label{obs:tagexchange}
If the $s$ spokes in the plafond triangulation have labels $d$, then their notched counterparts all have label $\frac{s}{d}$. More generally, if a radius has label $r_v$, then its notched counterpart has label $r_v^{\bowtie}=\frac{s}{d^2}r_v$.
\end{observation}

\begin{proof}
    For any pair of two distinct boundary vertices $v$ and $w$, we have two clock-hand rules, equating both $r_vr_w^{\bowtie}$ and $r_v^{\bowtie}r_w$ to a sum $\Sigma$ of labels of two complementary arcs: $$r_vr_w^{\bowtie}=r_v^{\bowtie}r_w=\Sigma.$$ Since all labels are positive, this implies $$\frac{r_v^{\bowtie}}{r_v}=\frac{\Sigma}{r_vr_w}=\frac{r_w^{\bowtie}}{r_w}:$$
    The fraction $\frac{r_v^{\bowtie}}{r_v}$ is thus independent of $v$. We will evaluate this fraction at a spoke. For the moment, assume $s\geq3$.  Consider the inner $s$-gon used in the plafond  triangulation, and label the vertices $u_1,\cdots,u_s$ clockwise. Denote by $a_i$ the label of the arc from $u_1$ to $u_{i+1}$ for $1\leq i\leq s-1$ going around the puncture clockwise. Note that $a_1=1$ by construction of the plafond triangulation. More generally, $a_i=i$ for $1\leq i\leq s-1$. Indeed, consider the quadritlateral with vertices $p,u_1,u_{i+1},u_{i+2}$. Then Ptolemy gives us $da_{i+1}=da_i+d$, since the spokes have label $d$, and we can induct. We can perform the analogous argument for the arc between $u_j$ and $u_{i+j}$ and show their label is $i$. In particular, the two arcs connecting two neighboring vertices $u_k$ and $u_{k+1}$ have labels $1$ and $s-1$. Applying the clock-hand rule, we get that $dr_u^{\bowtie}=1+(s-1)=s$ for any $u=u_j$, implying $r_u^{\bowtie}=\frac{s}{d}=\frac{s}{d^2}r_u$ since $r_u=d$, as desired. To treat the case $s=2$, apply the clock-hand rule to the digon, whose sides have label $1$, giving $dr_u^{\bowtie}=2,$ which is the same conclusion. When $s=1$, the spoke labels are $1$, and $d=1$, so the result follows.
\end{proof}

\section{Bounds on products of entries in distinct positions in \texorpdfstring{$A_m$}{Am}-chains}
In this section, we recall some results of Cheah and de Saint Germain who bound entries $a_k$ in $A_m$-chains by Virahanka-Fibonacci numbers: $a_k\leq F_{m+2}$, and generalize their result to bound products $a_k\cdot a_\ell$ of entries with $k\neq \ell$:  $a_k\cdot a_l\leq F_{m+2}F{m+1}$.

\begin{definition}
Let $m\geq 0$ be an integer. An \textrm{$A_m$-chain} is a sequence of positive integers $1=a_0,a_1,a_2,\ldots,a_m,a_{m+1}=1$ so that $a_i\mid a_{i-1}+a_{i+1}$ for all $1\leq i \leq m$.

We denote an $A_m$-chain more compactly by $(a_1,\ldots,a_m)$. (For $m=0$, this is the empty sequence.)
\end{definition}
By \cite[Lemma~2.4]{CheahSaintGermain2025}, cf. \cite{ConwayCoxeter1973a,ConwayCoxeter1973b}{ConwayCoxeter1973a}, $A_m$-chains correspond to diagonals of $A_m$-friezes.

Cheah and de Saint Germain proved the following\cite[Proposition 2.5(1)]{CheahSaintGermain2025}:
\begin{proposition}[Cheah--de Saint Germain lemma]\label{prop:cheadesaingermain} The entries $a_i$ of an $A_m$-chain satisfy the following inequality:
\begin{equation}\label{eq:single-bound}
 a_i\leq F_{m+2}.
\end{equation}
\end{proposition}

The crucial idea of Cheah and de Saint Germain produces an $A_{m-1}$-chain as follows:
\begin{enumerate}[label=\arabic*)]
\item If $a_m=1$, simply truncate the end and consider $(a_1,\ldots,a_{m-1})$.
\item If $a_m\ne1$, they show \cite[Lemma 2.6]{CheahSaintGermain2025} that for some index $r\in\{1,\ldots,m\}$
\begin{equation}\label{eq:iisspecial}
 a_r=a_{r-1}+a_{r+1}.
\end{equation}
(The letter $r$ is chosen because it is the \textbf{r}ight index.)
If we delete $a_r$ from the above $A_m$-chain, we get an $A_{m-1}$-chain. 
\end{enumerate}

Now note that in an $A_m$-chain, $a_k\cdot a_\ell\le F_{m+1}\cdot F_{m+2}$ as long as neither $k$ nor $\ell$ are $=r$ in the second scenario above (the first scenario is fine), by applying the bound \eqref{eq:single-bound}\ to the $A_m$-chain and to the $A_{m-1}$-chain (it becomes $\le F_{m+1}$). This holds even if $k$ or $\ell$ are allowed to be $=r$: 

\begin{lemma}[Cheah--de Saint Germain lemma for pairs ]\label{lem:pair}
For any $A_m$-chain $(a_1,\ldots,a_m)$ and $0\le k<\ell\le m+1$, we have $a_k\cdot a_\ell\le F_{m+1}F_{m+2}$.
\end{lemma}

\begin{proof} Assume for now that $m\geq2$. If $a_m=1$, we are in case 1 above and can reduce the lemma to the $A_{m-1}$-chain case. If  $a_m\neq 1$, then in view of the remark before the lemma and without loss of generality, let $\ell=r$. Then $a_k\cdot a_r=a_ka_{r-1}+a_ka_{r+1}$ by (\ref{eq:iisspecial}). 
If we knew the claim held for $A_{m-1}$-chains, then we would be fine in the case $r-1\neq k\ne r+1$. Indeed, the claim for $A_{m-1}$ would give
us $a_ka_{r-1}\le F_mF_{m+1}$ and $a_ka_{r+1}\le F_mF_{m+1}$, so that
\[
 a_k\cdot a_r\le(F_m+F_m)F_{m+1}
 \le(F_m+F_{m+1})F_{m+1}=F_{m+1}F_{m+2},
\]
as desired.

To handle the excluded case, let $k=r+1$. We have
\[
\begin{aligned}
 a_{r+1}a_r
 &=\underbrace{a_{r+1}a_{r-1}}_{\substack{\text{both live in the shortened chain, so}\\
       \text{Lemma \ref{lem:pair} for $A_{m-1}$ would give}\\
       \text{us that this is }\le F_mF_{m+1}}}
   +\underbrace{a_{r+1}^2}_{\substack{a_{r+1}\text{ is also in the shortened}\\
       \text{chain, so by Proposition \ref{prop:cheadesaingermain},}\\
       \text{this is }\le F_{m+1}^2}}\\[3pt]
 &\le(F_m+F_{m+1})F_{m+1}=F_{m+1}F_{m+2}. 
\end{aligned}
\]

The case $k=r-1$ is similar.
We thus showed  that the claim for $A_{m-1}$ implies the claim for $A_m$, so we reduce to the claim for $A_1$ by induction, but the claim for $A_1$ and $A_0$ is automatic\footnote{In case it is not, note that when $m=1$, then $a_1|1+1$, so $a_1=1$ or $a_1=2$. Both are at most $F_2F_3=2$. When $m=0$, $a_0=a_1=1$, and all this says is that $a_0\cdot a_1\leq F_1F_2,$ or $1\cdot 1\leq 1\cdot 1$. }. 
\end{proof}


\begin{corollary}\label{cor:ones}
Suppose that of the $m$ entries $a_1,\cdots,a_m$ in an $A_m$-chain $(a_1,\ldots,a_m)$, at least $h$ entries\footnote{So we are explicitly excluding the boundary entries $a_0$ and $a_{m+1}$ from the discussion} are $=1$. Let $0\le k<\ell\le m+1$. We then have $a_k,a_\ell\le F_{m+2-h}$ and $a_k\cdot a_\ell\le F_{m+1-h}F_{m+2-h}$.
\end{corollary}
\begin{proof} Without loss of generality, assume $h\leq m-1$.
We can cut the $A_m$-chain along each $a_i=1$ into at least $h+1$ subchains, each of length at most $m-h$. The first assertion follows from \cite[Proposition 2.5]{CheahSaintGermain2025}. As for the second assertion, if $a_k$ and $a_\ell$ lie in the same subchain, this follows from Lemma \ref{lem:pair}. If they lie in different subchains, then their lengths cannot both be $\ge m-h$, so without loss of generality, $a_k\le F_{m+1-h}$ by \cite[Proposition 2.5]{CheahSaintGermain2025} again. The corollary follows.
\end{proof}

\section{Bounds on \texorpdfstring{$D_n$}{Dn}-friezes coming from plafond triangulation with 1's on the spokes}\label{sec:unitary}

In this section, we assume that we are given a $D_n$-frieze with $n\geq4$ whose plafond triangulation has 1's on the $s$ spokes. As before, we label the vertices $v_1,\ldots,v_n$ clockwise and denote by $r_i$ the label of the radius connecting the puncture and $v_i$. Without loss of generality, assume the radius $r_1$ is a spoke.

For the moment, let $s\geq2$. We now cut the triangulation along $r_1$: double the vertex $v_1$, (call them $v_{n+1}$ and $v_1$), double the radius $r_1=1$ (call them $r_1$ and $r_{n+1}$), but keep the puncture $p$, which is now a(n outer) vertex. The Ptolemy relations around a vertex $v_j$ give
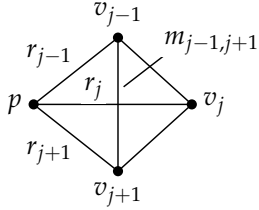
\begin{center}
\begin{minipage}{0.30\linewidth}\centering\radiiquadrilateral\end{minipage}\hfill
\begin{minipage}{0.67\linewidth}
\[
 r_j\times m_{j-1,j+1}=r_{j-1}\times1+r_{j+1}\times1.
\]
In particular, $r_j\mid r_{j-1}+r_{j+1}$ for $2\le j\le n$.
\end{minipage}
\end{center}
Thus, $(r_2,\ldots,r_n)$ is an $A_{n-1}$-chain with $r_i=1$ for at least $s-1$ of the $r_i$'s. 

For $s=1$, both the plain and notched radii $r_1$ and $r_1^{\bowtie}$ are part of the triangulation. We cut open $r_1$ as above. Now the notched $r_1^{\bowtie}$ comes from the loop around $p$ based at $v_1$, so that after cutting open, $r_1^{\bowtie}$ opens into an arc connecting $v_1$ and $v_{n+1}$ labeled $1$, cf. \cite[Remark~6.5, Definition~6.7, and Proof of Proposition~A.2]{BaurMarsh2009}. Consequently, the above argument with the Ptolemy relations applies also when $s=1$.

Thus from Corollary \ref{cor:ones} we obtain:
\begin{proposition}\label{prop:radii}
Choose $k<\ell\in\{1,\ldots,n+1\}$. Then $r_k,r_\ell\le F_{n+2-s}$\footnote{$(n-1)+2-(s-1)=n-s+2=n+2-s$.} and $r_kr_\ell\le F_{n+2-s}\cdot F_{n+1-s}$.
\end{proposition}
This bounds the labels of the plain arcs involving $p$ (i.e. the plain radii). It remains to bound the other plain arcs, and the notched radii.


For arcs connecting $v_i$ and $v_j$ ($i\ne j$), there are two possibilities as they can go around $p$ in two ways\footnote{When $i$ and $j$ differ by one, one of the edges is the boundary edge. This is why we chose to include these edges in the triangulation.}. Denote their labels by $a$ and $a'$.

By the clock-hand rule, we have $a+a'=r_ir_j^{\bowtie}=r_ir_js$ by Observation \ref{obs:tagexchange}

Since $a$ and $a'$ are integers $\ge1$, we have
\begin{equation}\label{eq:arc-bound}
 a\le sr_ir_j-1
 \underset{\text{Proposition \ref{prop:radii}}}{\le} sF_{n+1-s}F_{n+2-s}-1
 \underset{\text{Lemma \ref{lem:fibonacci}below}}{\le} F_nF_{n+1}-1.
\end{equation}

We justify the last inequality.
\begin{lemma}\label{lem:fibonacci}We have $$sF_{n+1-s}F_{n+2-s}\le F_nF_{n+1}.$$\end{lemma}
\begin{proof} For an integer $1\leq x\leq n-1$, set $\Phi(x):=xF_{n+1-x}F_{n+2-x}$. Note that $$\frac{\Phi(x+1)}{\Phi(x)}=\frac{x+1}{x}\times \frac{F_{n-x}}{F_{n+2-x}}\leq\frac{x+1}{x}\times\frac{1}{2}\leq 1,$$
since $F_{n+2-x}\geq 2F_{n-x}.$ Thus, $\Phi(s)\leq \Phi(1)=F_nF_{n+1}.$
\end{proof}

The only labels to be bounded are the notched radii.
Recall from Proposition \ref{prop:radii} that $r_k\le F_{n+2-s}$ and by Observation \ref{obs:tagexchange} together with our assumption that $d=1$, $r_k^{\bowtie}\le sF_{n+2-s}<F_nF_{n+1}$ when $n\ge4$:
This is automatic for $s=1$. If $s\ge2$, note $F_{n+2-s}\le F_n$ and $s\le n<F_{n+1}$. We have proved:
\begin{theorem}\label{thm:unitary}
Assume the $D_n$-frieze comes from a plafond triangulation with 1's on the spokes. Then each label is at most $F_nF_{n+1}-1$.
\end{theorem}

\section{The bound on any \texorpdfstring{$D_n$}{Dn}-frieze.}
Let $\fr$ be any $D_n$-frieze. By \cite[Proposition 3.2]{FontainePlamondon2016}, it has a plafond triangulation. Further, we can compare $\fr$ with the frieze $U_1$ (same triangulation, but spokes are labelled by 1) and $U_s$ (same triangulation, spokes labelled by $s$):
By \cite[Lemma 4.9 and sentence after its proof]{FontainePlamondon2016}, for any arc $a$ between boundary vertices has the same labels: $\fr(a)=U_1(a)=U_s(a)$, so that $\fr(a)\le F_nF_{n+1}-1$.
We would like to give the same estimate for radii. If $s=1$, there is nothing new to prove, as the only label for a spoke is 1. If $s\ge2$, choose any vertex $v$ and let $w$ be a vertex involving spokes and different from $v$, so that $U_1(r_w)=1$. By the clock-hand rule, if $a$ and $a'$ are the two arcs associated to $v$ and $w$ (going around either way around $p$), we have for any frieze $\G$:

\begin{center}\comparisonDigon $ \G(a)+\G(a')
 =\G(r_v)\G(r_w^{\bowtie})
 =\G(r_v^{\bowtie})\G(r_w)$.\end{center}

Thus,
\begin{equation}\label{eq:comparison}
\begin{aligned}
 U_1(r_v^{\bowtie})
 &=U_1(r_v^{\bowtie})U_1(r_w)
 =U_1(a)+U_1(a')\\
 &=\fr(a)+\fr(a')
 =\fr(r_v)\fr(r_w^{\bowtie}).
\end{aligned}
\end{equation}
Similarly, $U_1(r_v^{\bowtie})=\fr(r_v^{\bowtie})\fr(r_w)$.
Since $\fr(r_w)$ and $\fr(r_w^{\bowtie})$ are positive integers, i.e. $\geq1$, we have that $\fr(r_v),\fr(r_v^{\bowtie})\le U_1(r_v^{\bowtie})$.
Thus, any upper bound on $U_1(r_v^{\bowtie})$ is a bound on $\fr(r_v)$ and $\fr(r_v^{\bowtie})$.

\begin{corollary}\label{cor:D}
Let $n\geq4$ and $\fr$ be any $D_n$-frieze and $a$ be any entry. Then $\fr(a)\le F_nF_{n+1}-1$.
\end{corollary}

\begin{corollary}\label{cor:B}
Let $n\geq2$. Let $\fr$ be any $B_n$-frieze and $a$ be an entry. Then $\fr(a)\le F_{n+1}F_{n+2}-1$.
\end{corollary}
\begin{proof}
For $n=2$, this is \cite[Theorem 1(2)]{CheahSaintGermain2025}. ($B_2$ and $C_2$
are isomorphic.)
For $n\ge3$, every $B_n$-frieze can be lifted to a $D_{n+1}$-frieze, cf.\ \cite[Lemma 4.1 and Proof of Theorem 4.3]{FontainePlamondon2016}.
\end{proof}

\begin{corollary}\label{cor:mainresult}
    Robin Zhang's conjecture \cite[Conjecture 3]{Zhang2026} is true.
\end{corollary}
\begin{proof}
    Zhang's lower-bound construction \cite[Theorem 2]{Zhang2026} together with the above corollaries give the exact conjectured maxima. We note that for reasons of well-definedness, the requirements $n\geq4$ in the $D_n$-case and $n\geq 2$ in the $B_n$ case were made in \cite[Conjecture 3]{Zhang2026}, see \cite[Section~2.1]{Zhang2026}.
\end{proof}

\section*{Acknowledgments}
We thank Robin Zhang for explaining the conjecture and surrounding topics to the author while they both overlapped at IHES in July 2026. We thank Pierre-Guy Plamondon and also Bruce Fontaine for clarifying various questions we had about \cite{FontainePlamondon2016}. We thank the IHES and Harvard University for their hospitality and excellent working conditions. We found the proposed proof of Robin Zhang's conjecture in Office 525 in the Science Center.
The author was supported by Simons grant MPS-TSM-00008075.
\small
\bibliographystyle{plain}
\bibliography{friezesrzhangconj}
\end{document}